\documentclass[12pt]{article}
\usepackage{fullpage,comment,authblk,stackrel}
\usepackage[small]{caption}
\usepackage{tikz-cd}

\usepackage{nicefrac}
\usepackage{hyperref}
\usepackage{amssymb,mathtools,amsthm, amsmath}
\usepackage[all,cmtip]{xy}
\usepackage{eulervm,xfrac,mathdots}
\usepackage{rotating, setspace}

\newcommand{\blocktheorem}[1]{%
  \csletcs{old#1}{#1}
  \csletcs{endold#1}{end#1}
  \RenewDocumentEnvironment{#1}{o}
    {\par\addvspace{1.5ex}
     \noindent\begin{minipage}{\textwidth}
     \IfNoValueTF{##1}
       {\csuse{old#1}}
       {\csuse{old#1}[##1]}}
    {\csuse{endold#1}
     \end{minipage}
     \par\addvspace{1.5ex}}
}

\newcommand{\Rspace}        	{{\mathbb R}}

\newcommand{\Zspace}        	{{\mathbb Z}}

\newcommand{\mono}		{\hookrightarrow}

\newcommand{\Ccat}          	{{\mathcal{C}}}

\newcommand{\Pcat}          	{{\mathsf{P}}}
\newcommand{\Qcat}          	{{\mathsf{Q}}}

\newcommand{\Vect}          	{{\mathsf{Vec}}}

\newcommand{\Dgm}          	{{\mathsf{Dgm}}}

\newcommand{\Ab}      	    	{{\mathsf{Ab}}}

\newcommand{\Finab}      	    	{{\mathsf{FinAb}}}

\newcommand{\Ggroup}          	{{\mathcal{G}}}

\newcommand{\Ffunc}          	{{\mathsf{F}}}
\newcommand{\Gfunc}          	{{\mathsf{G}}}
\newcommand{\Hfunc}          	{{\mathsf{H}}}

\newcommand{\Kfunc}          	{{\mathsf{K}}}

\newcommand{\image}		{\mathsf{im}}

\newcommand{\id}			{\mathsf{id}}

\newcommand{\field}			{\mathsf{k}}

\newcommand{\ee}			{\varepsilon}

\newcommand\define[1]		{{\bf{#1}}}

\makeatletter
\def\moverlay{\mathpalette\mov@rlay}
\def\mov@rlay#1#2{\leavevmode\vtop{%
   \baselineskip\z@skip \lineskiplimit-\maxdimen
   \ialign{\hfil$\m@th#1##$\hfil\cr#2\crcr}}}
\newcommand{\charfusion}[3][\mathord]{
    #1{\ifx#1\mathop\vphantom{#2}\fi
        \mathpalette\mov@rlay{#2\cr#3}
      }
    \ifx#1\mathop\expandafter\displaylimits\fi}
\makeatother

\newtheoremstyle{amit}
{7pt}
{7pt}
{}
{7pt}
{\bf}
{:}
{.5em}
{}

\hypersetup{
  colorlinks   = true, 
  urlcolor     = blue, 
  linkcolor    = blue, 
  citecolor   = blue 
}

\begin{document}

\theoremstyle{amit}
\newtheorem{defn}{Definition}[section]
\newtheorem{prop}[defn]{Proposition}
\newtheorem{lem}[defn]{Lemma}
\newtheorem{thm}[defn]{Theorem}
\newtheorem{cor}[defn]{Corollary}
\newtheorem{rmk}[defn]{Remark}
\newtheorem{ex}[defn]{Example}

\newcommand\blfootnote[1]{%
  \begingroup
  \renewcommand\thefootnote{}\footnote{#1}%
  \addtocounter{footnote}{-1}%
  \endgroup
}

\title{Bottleneck Stability for Generalized Persistence Diagrams}
\author[1]{Alex McCleary}
\author[1]{Amit Patel}
\affil[1]{Department of Mathematics, Colorado State University}
\date{}

\maketitle

\abstract{In this paper, we extend bottleneck stability to the setting of one dimensional
constructible persistences module valued in any skeletally small abelian category.
}

\section{Introduction}

Persistent homology is a way of quantifying the topology of a function.
Given a function $f : X \to \Rspace$, persistence scans the homology of the sublevel
sets $f^{-1}(-\infty, r]$ as $r$ varies from $-\infty$ to $\infty$.
As it scans, homology appears and homology disappears.
This history of births and deaths is recorded as a \emph{persistence diagram} \cite{CSEdH}
or a \emph{barcode} \cite{ZC2005}.
What makes persistence special is that the persistence diagram of $f$ is stable
to arbitrary perturbations to $f$.
This is the celebrated \emph{bottleneck stability} of Cohen-Steiner, Edelsbrunner, and Harer \cite{CSEdH}.
Bottleneck stability makes persistent homology a useful tool in data analysis and in pure mathematics.
All of this is in the setting of vector spaces where each homology group is computed using coefficients in a field.

Fix a field $\field$ and let $\Vect$ be the category of $\field$-vector spaces.
As persistence scans the sublevel sets of $f$, it records its homology 
as a functor $\Ffunc : (\Rspace, \leq) \to \Vect$ where
$\Ffunc(r) := \Hfunc_\ast \big( f^{-1}(-\infty, r]; \field \big)$ and 
$\Ffunc( r \leq s) : \Ffunc(r) \to \Ffunc(s)$ is the map induced by the inclusion of the
sublevel set at $r$ into the sublevel set at $s$.
The functor $\Ffunc$ is called the \emph{persistence module} of~$f$.
Assuming some tameness conditions on $f$, the persistence diagram of $\Ffunc$ is equivalent to its barcode,
but the two definitions are very different.
The \emph{barcode} of $\Ffunc$ is its list of indecomposables.
This list is unique up to a permutation and furthermore, each indecomposable is an
\emph{interval persistence module} \cite{ZC2005, Carlsson2010, Boevey}.
The barcode model is how most people now think about persistence.
However in \cite{CSEdH} where bottleneck stability was first proved, the persistence diagram
is defined as a purely combinatorial object.
The \emph{rank function} of $\Ffunc$ assigns to each pair of values $r \leq s$
the rank of the map $\Ffunc(r \leq s)$.
The M\"obius inversion of the rank function is the \emph{persistence diagram} of $\Ffunc$.
Remarkably, these two very different approaches to persistence give equivalent answers.

The persistence diagram of \cite{CSEdH} 
easily generalizes \cite{patel} to the setting of constructible persistence modules valued
in any skeletally small abelian category $\Ccat$.
The rank function of such a persistence module records the image
of each $\Ffunc(r \leq s)$ as an element of the Grothendieck group
of $\mathcal{C}$.
Here we are using the Grothendieck group of an abelian category: this is the abelian group
with one generator for each isomorphism class of objects and one relation for each short exact sequence.
The persistence diagram of $\Ffunc$ is then the M\"obius inversion of this rank function.
A weak form of stability was shown in \cite{patel}.
In this paper, we prove bottleneck stability.
Our proof is an adaptation of the proofs of \cite{CSEdH} and~\cite{crazy_persistence}.

We were hoping that the M\"obius inversion model for persistence would lead to a 
good theory of persistence for multiparameter persistence modules 
$\Ffunc : (\Rspace^k, \leq) \to~\Ccat$ \cite{Carlsson2009,Lesnick2015}.
The M\"obius inversion applies to arbitrary finite posets.
Assuming some finiteness conditions on $\Ffunc$, we may define its persistence diagram 
as the M\"obius inversion of its rank function.
The proof of bottleneck stability presented in this paper requires positivity of the persistence diagram;
see Proposition \ref{prop:positivity}.
Unfortunately, there are simple examples of multiparameter persistence modules
whose persistence diagrams are not positive.
Therefore the proof presented here does not generalize to the multiparameter setting.
It seems that the M\"obius inversion model for persistence works well only in the setting of 
one-parameter constructible persistence modules.

%

\section{Persistence Modules}

Fix a skeletally small abelian category $\mathcal{C}$.
By skeletally small, we mean that the collection of isomorphism classes of objects in $\mathcal{C}$
is a set.
For example, $\mathcal{C}$ may be the category of finite dimensional $\mathsf{k}$-vector spaces, 
the category of finite abelian groups, or the category of finite length $R$-modules.
Let $\bar{\Rspace} := \Rspace \cup \{ \infty \}$ be the totally ordered set of real numbers with the point 
$\infty$ satisfying $p < \infty$ for all $p \in \Rspace$.
For any $p \in \bar \Rspace$, we let $\infty + p = \infty$.

\begin{defn}
A \define{persistence module} is a functor $\Ffunc : \Rspace \to \mathcal{C}$.
Let 
$$S=\{s_1 < s_2 < \cdots < s_k < \infty \} \subseteq \bar \Rspace$$ be a finite subset.
A persistence module $\Ffunc$ is \define{$S$-constructible} if it satisfies the following two conditions:
\begin{itemize}
\item For $p \leq q < s_1$, $\Ffunc(p \leq q) : 0 \to 0$ is the zero map.
\item For $s_i \leq p \leq q< s_{i+1}$, $\Ffunc(p \leq q)$ is an isomorphism.
\item For $s_k \leq p \leq q \leq \infty$, $\Ffunc(p \leq q)$ is an isomorphism.
\end{itemize}
\end{defn}

For example, let $f : M \to \Rspace$ be a Morse function on a compact manifold $M$.
The function $f$ filters $M$ by sublevel sets $M_{\leq r}^f := \{ p \in M \; |\; f(p) \leq r \}$.
For every $r \leq s$, $M_{\leq r}^f \subseteq M_{\leq s}^f$.
Now apply homology with coefficients in a finite abelian group.
The result is a persistence module of finite abelian groups that is constructible
with respect to the set of critical values of $f$ union $\{ \infty\}$.
If one applies homology with coefficients in a field $\mathsf{k}$, then the result
is a constructible persistence module of finite dimensional $\mathsf{k}$-vector spaces.
In topological data analysis, one usually starts with a constructible filtration of a finite simplicial complex.

There is a natural distance between persistence modules called the \emph{interleaving
distance} \cite{proximity}.
For any $\ee \geq 0$, let $\Rspace \times_\ee \{ 0, 1\}$ to be the poset 
$\big(  \Rspace \times \{0\} \big) \cup \big(  \Rspace \times\{1\} \big)$
where $(p,t) \leq (q,s)$ if 
\begin{itemize}
\item $t=s$ and $p \leq q$, or
\item $t \neq s$ and $p + \ee \leq q$.
\end{itemize}
Let $\iota_0,\iota_1: \Rspace \hookrightarrow  \Rspace \times_\ee \{ 0, 1\}$ be the poset maps
 $\iota_0: p \mapsto (p,0)$ and $\iota_1: p \mapsto(p,1)$.

\begin{defn}
\label{defn:interleaving}
An \define{$\ee$-interleaving} between two constructible persistence modules 
$\Ffunc$ and $\Gfunc$ is a functor~$\Phi$ that makes the following diagram commute 
up to a natural isomorphism:
\begin{equation}
\label{dgm:interleaving}
\begin{gathered}
\xymatrix{
&  \Rspace \times_\ee \{ 0, 1\} \ar@{-->}[dd]^{\Phi} \\
 \Rspace \ar@{^{(}->}[ru]^{\iota_0} \ar[rd]_{\Ffunc} &&  \Rspace \ar@{^{(}->}[lu]_{\iota_1} \ar[ld]^{\Gfunc} \\
& \mathcal{C}. &
}
\end{gathered}
\end{equation}
Two constructible persistence modules $\Ffunc$ and $\Gfunc$ are \define{$\ee$-interleaved} if there is an 
$\ee$-interleaving between them. 
The \define{interleaving distance} $d_I(\Ffunc,\Gfunc)$ between $\Ffunc$ and $\Gfunc$ is the infimum 
over all $\ee\geq 0$ 
such that $\Ffunc$ and $\Gfunc$ are $\ee$-interleaved. 
This infimum is attained since both $\Ffunc$ and $\Gfunc$ are constructible.
If $\Ffunc$ and $\Gfunc$ are not interleaved, then we let $d_I(\Ffunc,\Gfunc)=\infty$.
\end{defn}

\begin{prop}[Interpolation]
\label{prop:interpolation}
Let $\Ffunc$ and $\Gfunc$ be two $\ee$-interleaved constructible persistence modules.
Then there exists a one-parameter family of constructible persistence modules 
$\{ \Kfunc_t \}_{t \in [0,1]}$ 
such that $\Kfunc_0 \cong \Ffunc$, $\Kfunc_1 \cong \Gfunc$, and 
$d_I ( \Kfunc_t, \Kfunc_s ) \leq \ee |t-s|$.
\end{prop}
\begin{proof}
Let $\Ffunc$ and $\Gfunc$ be $\ee$-interleaved by $\Phi$ as in Definition \ref{defn:interleaving}. 
Define $ \Rspace \times_\ee [0,1]$ as the poset with the underlying set $ \Rspace\times[0,1]$ 
and $(p,t) \leq (q,s)$ whenever $p+\ee |t-s| \leq q$.
Note that $\Rspace \times_\ee \{ 0, 1\}$ naturally embeds into $ \Rspace \times_\ee [0,1]$ via 
$\iota:(p,t) \mapsto (p,t)$.
See Figure~\ref{fig:interpolation}.
Finding $\{\Kfunc_t \}_{t\in [0,1]}$ is equivalent to finding a functor $\Psi$ that makes the following diagram 
commute up to a natural isomorphism:
\begin{center}
\begin{tikzcd}
 \Rspace \times_\ee \{ 0, 1\} \ar[rr,"\Phi"]\ar[d,hook,"\iota"']	&& \Ccat\\
 \Rspace \times_\ee[0,1] . \ar[urr,"\Psi"',dashed]
\end{tikzcd}
\end{center}
This functor $\Psi$ is the right Kan extension of $\Phi$ along $\iota$ for which
we now give an explicit construction.
For convenience, let $\Pcat :=  \Rspace \times_\ee \{ 0, 1\}$ and
$\Qcat :=  \Rspace\times_\ee[0,1]$.
For $(p,t) \in \Qcat$, let $\Pcat \uparrow 
(p,t)$ be the subposet of $\Pcat$ consisting of all elements
$(p', t') \in \Pcat$ such that $(p, t) \leq (p',t')$.
The poset $\Pcat \uparrow (p,t)$, for any $p \in \Rspace$ and $t \notin \{ 0, 1 \}$, has two minimal elements: $(p + \ee t , 0)$ and $\big(p + \ee (1-t), 1 \big)$.
	\begin{figure}
	\centering
	\includegraphics[scale = 0.90]{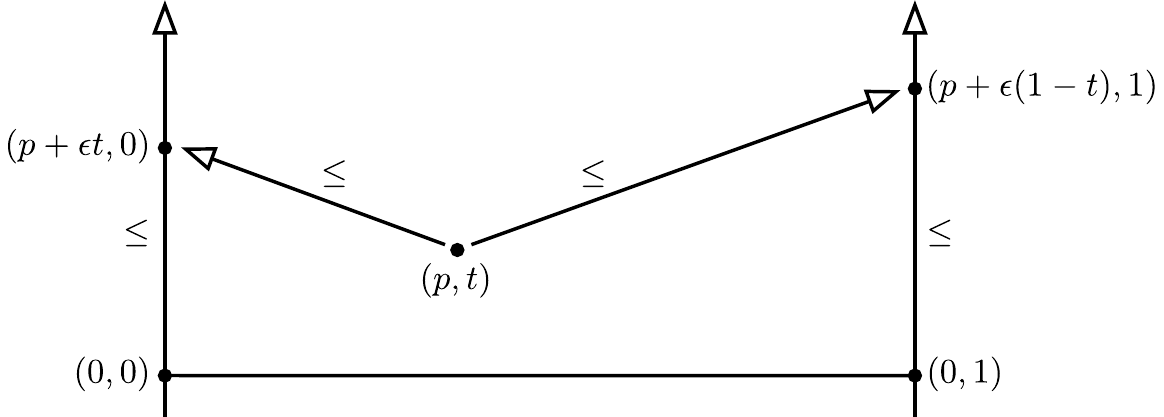}
	\caption{An illustration of the poset relation on $ \Rspace \times_\ee [0,1]$.}
	\label{fig:interpolation}
	\end{figure}
For $t \in \{ 0, 1 \}$, the poset $\Pcat \uparrow (p,t)$ has one minimal element,
namely $(p,t)$.
Let $\Psi \big( (p,t) \big) := \lim \Phi |_{\Pcat \uparrow (p,t)}$.
For $(p,t) \leq (q,s)$, the poset $\Pcat \uparrow (q,s)$ is a subposet of
$\Pcat \uparrow (p,t)$.
This subposet relation allows us to define the morphism $\Psi \big( (p,t) \leq (q,s) \big)$ 
as the universal morphism between the two limits.
Note that $\Psi \big( (p,0) \big)$ is isomorphic
to $\Ffunc(p)$ and $\Psi \big( (p,1) \big)$ is isomorphic to $\Gfunc(p)$.

We now argue that each persistence module $\Kfunc_t := \Psi(\cdot, t)$ is constructible.
As we increase $p$ while keeping $t$ fixed, the limit $\Kfunc_t(p)$ changes only when one of the 
two minimal objects of $\Pcat \uparrow (p,t)$ changes isomorphism type.
Since $\Ffunc$ and $\Gfunc$ are constructible, 
there are only finitely many such changes to the isomorphism type of $\Kfunc_t(p)$.
\end{proof}

\section{Persistence Diagrams}

Fix an abelian group $\Ggroup$ with a translation invariant partial ordering $\preceq$.
That is for all $a, b, c \in \Ggroup$, if $a \leq b$, then $a + c \leq b +c$.
Roughly speaking, a persistence diagram is the assignment to each interval
of the real line an element of $\Ggroup$.
In our setting, only finitely many intervals will have a nonzero value.

\begin{defn}
Let $\Dgm$ be the \define{poset of intervals} consisting of the following data:
	\begin{itemize}
	\item The objects of $\Dgm$ are intervals $[p,q) \subseteq \bar\Rspace$
	where $p \leq q$.
	\item The ordering is inclusion $[p_2,q_2) \subseteq [p_1,q_1)$.
	\end{itemize}
Given a finite set $S=\{s_1< s_2< \cdots <s_k < \infty\}\subseteq\bar{\Rspace}$, we use
$\Dgm(S)$ to denote the subposet of $\Dgm$ consisting of all intervals $[p,q)$ 
with $p,q \in S$. 
The \define{diagonal} $\Delta \subseteq \Dgm$ is the subset of intervals of the form 
$[p,p)$.
See Figure \ref{fig:dgm}.
\end{defn}

\begin{figure}
\begin{center}
\includegraphics[scale = 0.80]{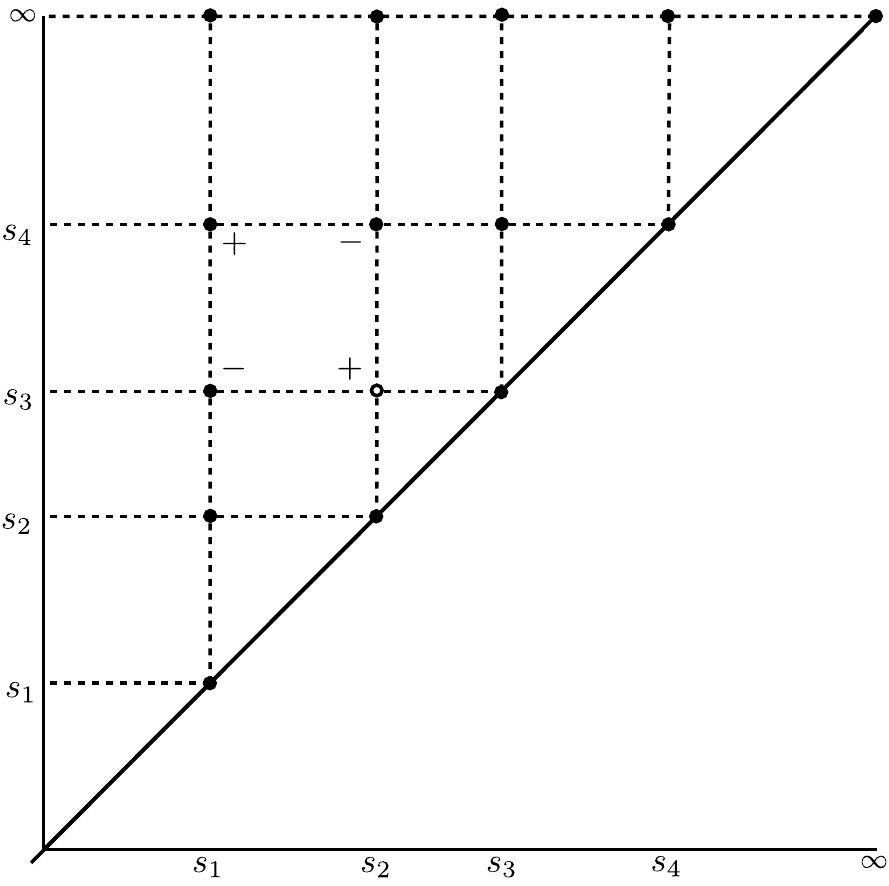}
\end{center}
\caption{An interval $I = [p,q)$ is visualized as the point $(p,q)$ in the plane.
The poset $\Dgm$ is therefore the set of points in the plane
on and above the diagonal.
In this example, $S  = \{ s_1 < s_2 < s_3 < s_4 < \infty\}$ and 
$\Dgm(S)$ is its set of grid points.
Given an $S$-constructible persistence module $\Ffunc$ and an interval $[s_2, s_3)$,
$\tilde{\Ffunc} \big( [s_2, s_3) \big) = d \Ffunc \big( [s_2, s_3) \big) - d \Ffunc \big( [s_2, s_4) \big) + d \Ffunc \big( [s_1, s_4) \big) - d \Ffunc \big( [s_1, s_3) \big)$.}
\label{fig:dgm}
\end{figure}

\begin{defn}
A \define{persistence diagram} is a map
$Y : \Dgm \to \Ggroup$ with finite support.
That is, there are only finitely many intervals $I \in \Dgm$
such that $Y(I) \neq 0$.
\end{defn}

We now introduce the bottleneck distance between persistence diagrams.

\begin{defn}
A \define{matching} between two persistence diagrams $Y_1, Y_2 : \Dgm \to \Ggroup$ is a 
map $\gamma: \Dgm \times \Dgm \to \Ggroup$ 
satisfying 
	\begin{align*}
	Y_1(I) & = \sum_{J \in \Dgm} \gamma(I,J) \textrm{ for all }I \in \Dgm \backslash \Delta \\
	Y_2(J) & = \sum_{I \in \Dgm} \gamma(I,J) \textrm{ for all } J \in \Dgm \backslash \Delta.
	\end{align*}
The \define{norm} of a matching $\gamma$ is
$$ ||\gamma|| := \max_{\big \{I=[p_1,q_1),J=[p_2,q_2) \,\big |\, \gamma(I,J)\neq 0 \big \}} 
\big \{|p_1-p_2|,|q_1-q_2| \big\}. $$
If either $q_1$ or $q_2$ is $\infty$, then $ | q_1 - q_2 | = \infty$.
The \define{bottleneck distance} between two persistence diagrams $Y_1, Y_2 : \Dgm \to \Ggroup$ is 
$$d_B(Y_1,Y_2) := \inf_\gamma ||\gamma||$$ 
over all matchings $\gamma$ between $Y_1$ and $Y_2$.
This infimum is attained since persistence diagrams have finite support.
\end{defn}

\section{Diagram of a Module}
We now describe the construction of a persistence diagram
from a constructible persistence module.
Fix a skeletally small abelian category $\Ccat$.

\begin{defn}
The \define{Grothendieck group} $\Ggroup(\mathcal{C})$ of $\mathcal{C}$ is the abelian group 
with one generator for each 
isomorphism class $[a]$ of objects $a \in \mathsf{ob}\; \mathcal{C}$ and one relation
$[b] = [a] + [c]$
for each short exact sequence
$0 \to a \to b \to c \to 0.$
The Grothendieck group has a natural translation invariant partial ordering where
$[a] \preceq [b]$ whenever $a\hookrightarrow b$.
For each $a\hookrightarrow b$, we have 
$a \oplus c \hookrightarrow b \oplus c$ for any object $c$ in $\mathcal{C}$.
This makes $\preceq$ a translation invariant partial ordering.
\end{defn}

\begin{ex}
Here are three examples of $\mathcal{C}$ with their Grothendieck groups.
	\begin{itemize}
	\item Let $\Vect$ be the category of finite dimensional $\field$-vector spaces
	for some fixed field~$\field$.
	The isomorphism class of a finite dimensional $\field$-vector space is completely determined
	by its dimension. 
	This means that the free abelian group generated by the set of isomorphism classes in 
	$\Vect$ is $\bigoplus_{n} \Zspace$ where $n \geq 0$ is a natural number.
	Since every short exact sequence in $\Vect$ splits, the only relations are of the form $[A] + [B] = [C]$ whenever $A \oplus B \cong C$.
	Therefore $\Ggroup (\Vect) \cong \Zspace$ where the translation invariant partial 
	ordering $\preceq$ is the usual total ordering on the integers.
	
	\item Let $\Finab$ be the category of finite abelian groups.
	A finite abelian group is isomorphic to
		$$\dfrac{\Zspace}{p_1^{n_1} \Zspace} \oplus \cdots \oplus
		 			\dfrac{\Zspace}{p_k^{n_k} \Zspace}$$
	where each $p_i$ is prime.
	The free abelian group generated by the set of isomorphism classes in $\Finab$
	is $\bigoplus_{(p,n)} \Zspace$
	over all pairs $(p, n)$ where $p$ is prime and $n \geq 0$ a natural number.
	Every primary cyclic group $\dfrac{\Zspace}{p^n \Zspace}$ fits into a short exact sequence
		\begin{equation*}
		\begin{tikzcd}
		0 \arrow[r] & \dfrac{\Zspace}{p^{n-1} \Zspace} \arrow[r, "\times p"] & 
		\dfrac{\Zspace}{p^{n} \Zspace} \arrow[r, "/"] &
		\dfrac{\Zspace}{p\Zspace} \arrow[r] & 0
		\end{tikzcd}
		\end{equation*}
	giving rise to a relation
	$\left[ \dfrac{\Zspace}{p^{n} \Zspace} \right] = 
	\left[ \dfrac{\Zspace}{p^{n-1} \Zspace} \right] + \left[ \dfrac{\Zspace}{p \Zspace} \right].$
	By induction, $\left[ \dfrac{\Zspace}{p^{n} \Zspace} \right] = n \left[ \dfrac{\Zspace}{p \Zspace} \right]$.
	Therefore $\Ggroup( \Finab ) \cong \bigoplus_p \Zspace$ where $p$ is prime.
	For two elements $[a], [b] \in \Ggroup( \Finab )$, $[a] \preceq [b]$ 
	if the multiplicity of each prime factor of $[a]$ is at most the multiplicity of each
	prime factor of $[b]$.
	
	\item
	Let $\Ab$ be the category of finitely generated abelian groups.
	A finitely generated abelian group is isomorphic to
		$$ \Zspace^m \oplus \dfrac{\Zspace}{p_1^{n_1} \Zspace} \oplus \cdots \oplus
		 			\dfrac{\Zspace}{p_k^{n_k} \Zspace}$$
	where each $p_i$ is prime.
	The free abelian group generated by the set of isomorphism classes in $\Ab$
	is $\Zspace \oplus \bigoplus_{(p,n)} \Zspace$
	over all pairs $(p, n)$ where $p$ is prime and $n \geq 0$ a natural number.
	In addition to the short exact sequences in $\Finab$, we have
		\begin{equation*}
		\begin{tikzcd}
		0 \arrow[r] & \Zspace \arrow[r, "\times p"] & 
		\Zspace \arrow[r, "/"] &
		\dfrac{\Zspace}{p\Zspace} \arrow[r] & 0
		\end{tikzcd}
		\end{equation*}
	giving rise to the relation $\left[ \dfrac{\Zspace}{p \Zspace} \right] = [0]$.
	Therefore $\Ggroup( \Ab) \cong \Zspace$ where $\preceq$ is the usual total ordering on the integers.
	Unfortunately all torsion is lost.
	\end{itemize}
\end{ex}

Given a constructible persistence module, we now record the images of all its maps
as elements of the Grothendieck group.

\begin{defn}
Let $S = \{s_1 < \cdots < s_k < \infty\}$ be a finite set and $\Ffunc$ an 
$S$-constructible persistence module valued in $\mathcal{C}$.
Choose a $\delta > 0$ such that $s_{i-1} < s_i - \delta$, for all $1< i \leq k$.
The \define{rank function} of $\Ffunc$ is the map
$d \Ffunc : \Dgm \to \Ggroup( \mathcal{C} )$ defined as follows:
$$d \Ffunc (I) =
\begin{cases}
\big[ \image \, \Ffunc(p<s_i-\delta) \big]	& \text{for } I = [p,s_i) \text{ where } 1 \leq i \leq k \\
\big[ \image \, \Ffunc(p \leq q) \big]	& \text{for all other } I=[p,q).
\end{cases}
$$
Note that for any $[p,q) \in \Dgm$, $d \Ffunc \big ( [p,q) \big)$ equals
$d \Ffunc(I)$ where $I$ is the largest interval in $\Dgm(S)$ containing $[p,q)$.
This means that $d \Ffunc$ is uniquely determined by its value on $\Dgm(S)$.
\end{defn}

\begin{prop}
Let $\Ffunc$ be a constructible persistence module valued in a skeletally small abelian category
$\mathcal{C}$.
Then its rank function $d \Ffunc : \Dgm \to \Ggroup(\mathcal{C})$ is a poset
reversing map.
That is for any pair of intervals 
$[p_2, q_2) \subseteq [p_1,q_1)$, 
$d \Ffunc \big( [p_1, q_1) \big) \preceq d \Ffunc \big( [p_2,q_2) \big)$.
\end{prop}
\begin{proof}
Suppose $\Ffunc$ is $S = \{s_1 < \cdots < s_k < \infty\}$-constructible.
Consider the following commutative diagram:
	\begin{equation*}
	\xymatrix{
	\Ffunc(p_1) \ar[d]_{h := \Ffunc (p_1 \leq q_1)} \ar[rrr]^{e := \Ffunc(p_1 \leq p_2)} 
	&&& \Ffunc(p_2) \ar[d]^{f := \Ffunc(p_2 \leq q_2)} \\
	\Ffunc(q_1) &&& \Ffunc(q_2). \ar[lll]^{g := \Ffunc(q_2 \leq q_1)}
	}
	\end{equation*}
We may assume $q_1, q_2 \notin S$.
If this is not the case, replace $q_1$ and/or $q_2$ in the above diagram with 
$q_1 - \delta$ and $q_2 - \delta$ for some sufficiently small $\delta > 0$.
We have $d \Ffunc \big( [p_1, q_1) \big) = [ \image\, h ]$ and
$d \Ffunc \big( [p_2, q_2) \big) =  [ \image\, f ]$.
Let $I := \image\, ( f \circ e )$ and $K := I \cap \ker g$.
Then $K \mono I \mono \image\, f$ and $\image\, h \cong \sfrac{I}{K}$.
Therefore $d \Ffunc \big ( [p_1, q_1) \big) \preceq d \Ffunc \big( [p_2, q_2 ) \big)$.
\end{proof}

Given the rank function $d \Ffunc : \Dgm \to \Ggroup(\mathcal{C})$
of an  $S$-constructible persistence module~$\Ffunc$, there is a unique
map $\tilde \Ffunc : \Dgm \to \Ggroup$ such that
	\begin{equation}
	\label{eq:mobius}
	d \Ffunc (I) =  \sum_{J \in \Dgm: J \supseteq I } \tilde \Ffunc (J)
	\end{equation}
for each $I \in \Dgm$.
This equation is the \emph{M\"obius inversion formula}.
For each $I = [s_i, s_j)$ in $\Dgm(S)$,

	\begin{equation}
	\label{eq:inversion}
	\tilde \Ffunc (I) = d \Ffunc  \big( [s_i, s_j) \big) - d \Ffunc \big( [s_{i}, s_{j+1}) \big)
	+ d \Ffunc \big( [s_{i-1}, s_{j+1}) \big) - d \Ffunc \big( [s_{i-1}, s_j) \big).
	\end{equation}
For each $I = [s_i, \infty)$ in $\Dgm(S)$,
	\begin{equation}
	\label{eq:inversion_infty}
	\tilde \Ffunc (I) = d \Ffunc  \big( [s_i, \infty) \big) - d \Ffunc \big( [s_{i-1}, \infty) \big).
	\end{equation}
For all other $I \in \Dgm$, $\tilde \Ffunc(I) = 0$.
Here we have to be careful with our indices.
It is possible $s_{j+1}$ or $s_{i-1}$ is not in $S$.
If $s_{j+1}$ is not in $S$, let $s_{j+1} = \infty$.
If $s_{i-1}$ is not in $S$, let $s_{i-1}$ be any value strictly less than $s_1$.
We call $\tilde \Ffunc$ the \emph{M\"obius inversion} of $d \Ffunc$.

\begin{defn}
\label{defn:diagram}
The \define{persistence diagram of a constructible persistence module} $\Ffunc$
is the M\"obius inversion $\tilde \Ffunc$ of its rank function $d \Ffunc$.
\end{defn}

The Grothendieck group of $\Ccat$ has one relation for each short exact sequence in $\Ccat$.
These relations ensure that the persistence diagram of a persistence module is positive which
plays a key role in the proof of Lemma \ref{lem:box}.

\begin{prop}[Positivity \cite{patel}]
\label{prop:positivity} 
Let $\Ffunc$ be a constructible persistence module 
valued in a skeletally small abelian category $\Ccat$.
Then for any $I\in \Dgm$, we have~$[0] \preceq \tilde \Ffunc(I)$.
\end{prop}

\section{Stability}
We now begin the task of proving bottleneck stability.
Throughout this section, persistence modules are valued in a fixed skeletally small abelian
category $\Ccat$.

\begin{defn}
For an interval $I = [p,q)$ in $\Dgm$ and a value $\ee \geq 0$, let
$$ \square_{\ee} I :=  \big\{ [r,s) \in \Dgm \, \big |\, p-\ee < r \leq p+\ee \text{ and } q-\ee \leq s < q+ \ee \big \} $$ 
be the subposet of $\Dgm$ consisting of intervals $\ee$-close to $I$.
If $I$ is too close to the diagonal, that is if $q - \ee \leq p+\ee$, then we let
$\square_\ee I$ be empty.
We call $\square_\ee I$ the $\ee$-\emph{box} around $I$.
See Figure \ref{fig:integration}.
Note that if $q = \infty$, then $\square_\ee I = \big \{ [r, \infty) \; \big | \;  p-\ee < r \leq p+ \ee \big \}$.
\end{defn}

	\begin{figure}
	\begin{center}
	\includegraphics[scale = 0.80]{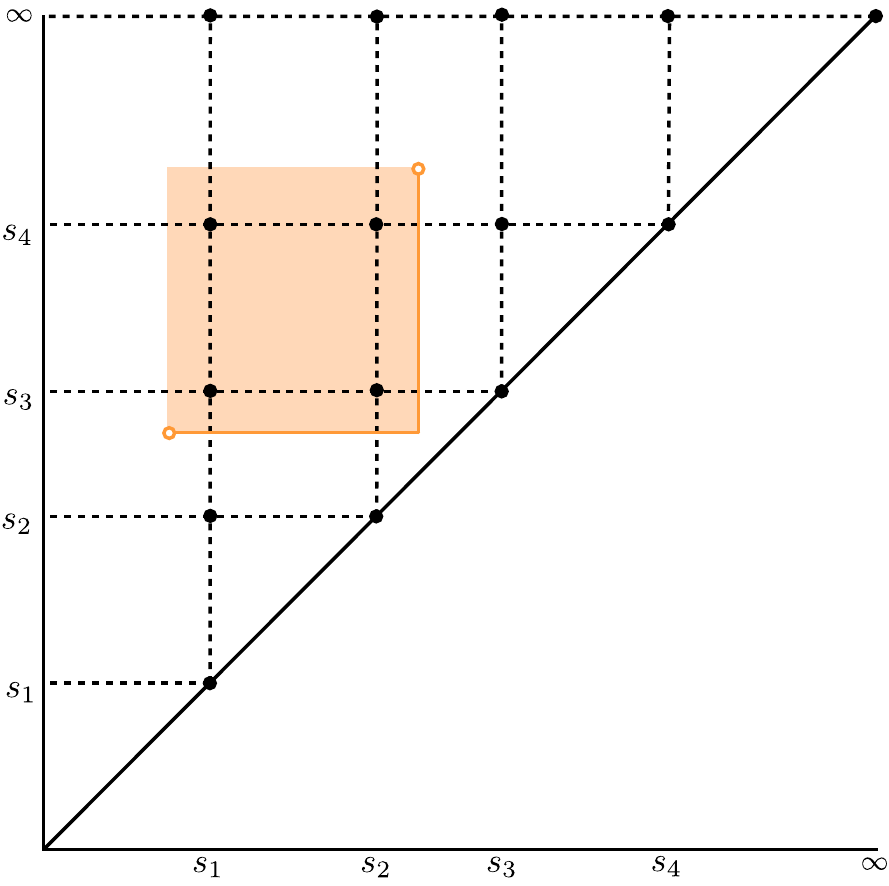}
	\end{center}
	\caption{The shaded area is a box $\square_\ee I$.
	Note that $\square_\ee I$ is closed on the bottom and right, and it is open on the top and left.
	}
	\label{fig:integration}
	\end{figure}

\begin{lem}
\label{lem:corners}
Let $\Ffunc$ be an $S$-constructible persistence module, $I = [p,q)$, and $\ee > 0$.
If $\square_\ee I$ is nonempty, then
	\begin{equation*}
	\label{eq:sum}
	\sum_{J \in \square_\ee I} \tilde{F}(J) = d \Ffunc \big( [p+\ee, q-\ee) \big) - 
	d \Ffunc \big( [p+\ee, q+\ee) \big) 
	+ d \Ffunc \big( [p-\ee, q+\ee) \big) - d \Ffunc \big( [p-\ee, q-\ee) \big)
	\end{equation*}
whenever $q \neq \infty$ and
	\begin{equation*}
	\label{eq:sum}
	\sum_{J \in \square_\ee I} \tilde{F}(J) = d \Ffunc \big( [p+\ee, \infty) \big) 
	- d \Ffunc \big( [p-\ee, \infty) \big)
	\end{equation*}
whenever $q = \infty$.
\end{lem}

\begin{proof}
Both equalities follow easily from the M\"obius inversion formula; see Equation~\ref{eq:mobius}.
If $q \neq \infty$, then
	\begin{align*}
	\sum_{J \in \square_\ee I} \tilde \Ffunc(J) &= \sum_{\substack{J \in \Dgm: \\ J \supseteq [p+\ee,q-\ee)}} 
	\tilde \Ffunc(J)
	- \sum_{\substack{J \in \Dgm: \\ J \supseteq [p+\ee, q+\ee)}} \tilde \Ffunc(J) +
	\sum_{\substack{J \in \Dgm: \\ J \supseteq [p-\ee, q+\ee)}} \tilde \Ffunc(J) - 
	\sum_{\substack{J \in \Dgm: \\ J \supseteq [p-\ee, q-\ee)} } \tilde \Ffunc(J) \\
	& = d \Ffunc \big( [p+\ee,q-\ee) \big) - d \Ffunc \big( [p+\ee,q+\ee) \big) - 
	d \Ffunc \big( [p-\ee,q+\ee) \big) + d \Ffunc \big( [p-\ee,q-\ee) \big).
	\end{align*}
If $q = \infty$, then
	\begin{align*}
	\sum_{J \in \square_\ee I} \tilde \Ffunc(J) &= \sum_{\substack{J \in \Dgm: \\ J 
	\supseteq [p+\ee,\infty)}} \tilde \Ffunc(J)
	- \sum_{\substack{J \in \Dgm: \\ J \supseteq [p-\ee, \infty)}} \tilde \Ffunc(J) \\
	& = d \Ffunc \big( [p+\ee,\infty) \big) - d \Ffunc \big( [p-\ee,\infty) \big).
	\end{align*}
\end{proof}


\begin{lem}[Box Lemma]
\label{lem:box}
Let $\Ffunc$ and $\Gfunc$ be two $\ee$-interleaved constructible persistence modules, $I \in \Dgm$, 
and $\mu > 0$.
Then 
$$\sum_{J \in \square_{\mu} I} \tilde \Ffunc(J) \preceq \sum_{J \in \square_{\mu + \ee} I} \tilde \Gfunc(J)$$
whenever $\square_{\mu + \ee} I$ is nonempty.
\end{lem}

\begin{proof}
Suppose $\Ffunc$ and $\Gfunc$ are
$\ee$-interleaved by $\Phi$ in Diagram \ref{dgm:interleaving}. 
Define $\varphi_{r}: \Ffunc(r) \to \Gfunc(r+\ee)$ as 
$\Phi \big( (r,0) \leq (r+\ee,1) \big)$ and define 
$\psi_{r}: \Gfunc (r) \to \Ffunc (r+\ee)$ as $\Phi \big( (r,1) \leq (r+\ee,0) \big)$.

Suppose $I = [p,q)$ where $q \neq \infty$. By Lemma \ref{lem:corners}, 
\begin{align*}
\sum_{J \in \square_{\mu} I } \tilde{F}(J) =\; & 
	d \Ffunc \big( [p+\mu, q-\mu) \big) - d \Ffunc \big( [p+\mu, q+\mu) \big) \\
	&+ d \Ffunc \big( [p-\mu, q+\mu) \big) - d \Ffunc \big( [p-\mu, q-\mu) \big) \\
\sum_{J \in \square_{\mu+\ee} I} \tilde{G}(J) =\; & 
	d \Gfunc \big( [p+\mu+\ee, q-\mu-\ee) \big) - d \Gfunc \big( [p+\mu+\ee, q+\mu+\ee) \big) \\
	& + d \Gfunc \big( [p-\mu-\ee, q+\mu+\ee) \big) - d \Gfunc \big( [p-\mu-\ee, q-\mu-\ee) \big).
\end{align*}
Choose a sufficiently small $\delta > 0$ so that we have the following equalities:
	\begin{align*}
d \Ffunc \big( [p+\mu, q-\mu) \big) &= \big[ \image\; \Ffunc ( p + \mu < q-\mu-\delta) \big] \\  
d \Ffunc \big( [p+\mu, q+\mu) \big) &= \big[ \image\; \Ffunc ( p+\mu < q+\mu-\delta) \big] \\
d \Ffunc \big( [p-\mu, q+\mu) \big) &= \big[ \image\; \Ffunc ( p-\mu < q+\mu-\delta) \big] \\
d \Ffunc \big( [p-\mu, q-\mu) \big) &= \big[ \image\; \Ffunc ( p-\mu < q-\mu-\delta) \big] \\
d \Gfunc \big( [p+\mu+\ee, q-\mu-\ee) \big) &= \big[ \image\; \Gfunc ( p+\mu+\ee < q-\mu-\ee-\delta) \big] \\
d \Gfunc \big( [p+\mu+\ee, q+\mu+\ee) \big) &= \big[ \image\; \Gfunc ( p+\mu+\ee < q+\mu+\ee-\delta) \big]\\
d \Gfunc \big( [p-\mu-\ee, q+\mu+\ee) \big) &= \big[ \image\; \Gfunc ( p-\mu-\ee < q+\mu+\ee-\delta) \big] \\
d \Gfunc \big( [p-\mu-\ee, q-\mu-\ee) \big) &= \big[ \image\; \Ffunc ( p-\mu-\ee < q-\mu-\ee-\delta) \big].
	\end{align*}
Consider the following commutative diagram where the horizontal and vertical arrows are the appropriate 
morphisms from $\Ffunc$ and $\Gfunc$:
\begin{center}
\begin{tikzcd}
\Gfunc(p-\mu-\ee) \ar[rrrr] \ar[dddd] \ar[dr,"\psi_{p-\mu -\ee}"]
&&&& G(p+\mu+\ee) \ar[dddd] \\
& \Ffunc (p-\mu) \ar[rr] \ar[dd]&& \Ffunc(p+\mu) \ar[dd] \ar[ur,"\varphi_{p+\mu}"]\\
\\
  &\Ffunc(q+\mu-\delta) \ar[dl,"\varphi_{q+\mu-\delta}"]&& \Ffunc(q-\mu-\delta) \ar[ll] \\
\Gfunc(q+\mu+\ee-\delta)	&&&& \Gfunc(q-\mu-\ee-\delta). \ar[llll] \ar[ul,"\psi_{q-\mu-\ee-\delta}"]
\end{tikzcd}
\end{center}
Choose two values $a < b$ such that $a+ \mu+\ee < b-\mu-\ee$ and let
$$T := \{ a-\mu-\ee < a-\mu < a+\mu < a+\mu+\ee < c < b-\mu-\ee < b-\mu < b+\mu < \infty \} 
\subseteq \bar \Rspace.$$
Let $\Hfunc: \Rspace \to \mathcal{C}$ be the $T$-constructible persistence module determined by the following diagram:
\begin{equation*}
\xymatrix{
\Hfunc(a-\mu-\ee) = \Gfunc(p-\mu-\ee) \ar[r] & \Hfunc(a-\mu) = \Ffunc(p-\mu) \ar[r] & 
\Hfunc(a+\mu) = \Ffunc(p+\mu) \ar[d] \\
\Hfunc(b-\mu-\ee) = \Ffunc(q-\mu-\delta) \ar[d] & \Hfunc(c) = \Gfunc (q-\mu-\ee-\delta) \ar[l] &  
\Hfunc(a+\mu+\ee) = \Gfunc(p+\mu+\ee) \ar[l] \\
\Hfunc(b-\mu) = \Ffunc(q+\mu-\delta) \ar[r] & \Hfunc(b+\mu) = \Gfunc(q+\mu+\ee-\delta).
}
\end{equation*}
Here the value of $\Hfunc$ is given on each value in $T$ and morphisms between adjacent objects are the connecting
morphisms in the above commutative diagram.
For example, for all $a+\mu+\ee \leq r < c$, $\Hfunc(r) = \Gfunc(p+\mu+\ee)$ and 
$\Hfunc(a+\ee+\mu \leq r) = \id$.
The morphism $\Hfunc(c \leq b-\mu-\ee)$ is $\psi_{q-\mu-\ee-\delta}$.
We have the following equalities:
\begin{align*}
\big[ \image\; \Ffunc ( p + \mu < q-\mu-\delta) \big] &= d \Hfunc \big( [a+\mu, b-\mu) \big) \\
\big[ \image\; \Ffunc ( p + \mu < q+\mu-\delta) \big] &= d \Hfunc \big( [a+\mu, b+\mu) \big)  \\
\big[ \image\; \Ffunc ( p - \mu < q+\mu-\delta) \big] &= d \Hfunc \big( [a-\mu, b+\mu) \big) \\
\big[ \image\; \Ffunc ( p - \mu < q-\mu-\delta) \big] &= d \Hfunc \big( [a-\mu, b-\mu) \big) \\
\big[ \image\; \Gfunc ( p+\mu+\ee < q-\mu-\ee-\delta) \big] &= d \Hfunc \big( [a+\mu+\ee, b-\mu-\ee) \big) \\ 
\big[ \image\; \Gfunc ( p+\mu+\ee < q+\mu+\ee-\delta) \big] &= 
	d \Hfunc \big( [a+\mu+\ee, b+\mu+\ee) \big) \\
\big[ \image\; \Gfunc ( p-\mu-\ee < q+\mu+\ee-\delta) \big] &= d \Hfunc \big( [a-\mu-\ee, b+\mu+\ee) \big) \\
\big[ \image\; \Gfunc ( p-\mu-\ee < q-\mu-\ee-\delta) \big] &= d \Hfunc \big( [a-\mu-\ee, b-\mu-\ee) \big).
\end{align*}
By Lemma \ref{lem:corners} along with the above substitutions, we have
$$\sum_{J \in \square_{\mu}[a,b)}  \tilde \Hfunc (J) = \sum_{J \in \square_{\mu} I} \tilde{\Ffunc}(J)$$
$$\sum_{J \in \square_{\mu+\ee}[a,b)} \tilde \Hfunc (J) = \sum_{J \in \square_{\mu+\ee} I } \tilde{G}(J).$$ 
By the inclusion $\square_{\mu} [a,b) \subseteq \square_{\mu + \ee} [a,b) $ along with Proposition \ref{prop:positivity}, we have 
$$\sum_{J \in \square_{\mu} [a,b)} \tilde \Hfunc (J) \preceq 
\sum_{J \in \square_{\mu+\ee} [a,b)} \tilde \Hfunc (J).$$
This proves the statement.

Suppose $I = [p,\infty)$ and let $z \in \Rspace$ be larger than $\mu + \ee$ and $s_k$ where $\Ffunc$ and $\Gfunc$ are $\{ s_1 < \cdots < s_k < \infty \}$-constructible.
By Lemma \ref{lem:corners}, 
\begin{align*}
\sum_{J \in \square_{\mu} I } \tilde{F}(J) =\; & 
	d \Ffunc \big( [p+\mu, \infty) \big) - d \Ffunc \big( [p-\mu, \infty) \big) \\
\sum_{J \in \square_{\mu+\ee} I} \tilde{G}(J) =\; & 
	d \Gfunc \big( [p+\mu+\ee, \infty) \big) - d \Gfunc \big( [p-\mu-\ee, \infty) \big).
\end{align*}
We have the following equalities:
	\begin{align*}
d \Ffunc \big( [p+\mu, \infty) \big) &= \big[ \image\; \Ffunc ( p+\mu < \infty) \big] \\  
d \Ffunc \big( [p-\mu, \infty) \big) &= \big[ \image\; \Ffunc ( p-\mu < \infty) \big] \\
d \Gfunc \big( [p+\mu+\ee, \infty) \big) &= \big[ \image\; \Gfunc ( p+\mu+\ee < \infty) \big] \\
d \Gfunc \big( [p-\mu-\ee, \infty) \big) &= \big[ \image\; \Gfunc ( p-\mu-\ee < \infty) \big].
	\end{align*}
Consider the following commutative diagram where the vertical and horizontal arrows are the appropriate
morphisms from $\Ffunc$ and $\Gfunc$:
\begin{center}
\begin{tikzcd}
\Gfunc(p-\mu-\ee) \ar[rrrr] \ar[dddd]  \ar[dr,"\psi_{p-\mu-\ee}"] &&&& \Gfunc(p+\mu+\ee) \ar[dddd] \\
& \Ffunc (p-\mu) \ar[rr] \ar[dd] && \Ffunc(p+\mu) \ar[dd] \ar[ur,"\varphi_{p+\mu}"] & \\
&&&& \\
& \Ffunc(\infty) \ar[dl,"\varphi_{\infty}"]&& \Ffunc(\infty) \ar[ll]  & \\
\Gfunc(\infty) &&&& \Gfunc(\infty). \ar[llll]  \ar[ul,"\psi_{\infty}"]
\end{tikzcd}
\end{center}
Note that $\psi_\infty$ and $\varphi_\infty$ are isomorphisms.
Let
$$T := \{ -\mu-\ee < -\mu < \mu < \mu+\ee < \infty \} 
\subseteq \bar \Rspace.$$
Let $\Hfunc: \Rspace \to \mathcal{C}$ be the $T$-constructible persistence module determined by the following diagram:
\begin{equation*}
\xymatrix{
\Hfunc(-\mu-\ee) = \Gfunc(p-\mu-\ee) \ar[r] & \Hfunc(-\mu) = \Ffunc(p-\mu) \ar[r] & 
\Hfunc(\mu) = \Ffunc(p+\mu) \ar[d] \\
& \Hfunc(z) = \Ffunc (z)  &  \Hfunc(\mu+\ee) = \Gfunc(p+\mu+\ee). \ar[l]
}
\end{equation*}
Here the value of $\Hfunc$ is given on each value in $T$ and morphisms between adjacent objects are the connecting
morphisms in the above commutative diagram.
We have the following equalities:
\begin{align*}
\big[ \image\; \Ffunc ( p + \mu <\infty) \big] &= d \Hfunc \big( [\mu, \infty) \big) \\
\big[ \image\; \Ffunc ( p - \mu < \infty) \big] &= d \Hfunc \big( [-\mu, \infty) \big) \\
\big[ \image\; \Gfunc ( p+\mu+\ee < \infty) \big] &= d \Hfunc \big( [\mu+\ee, \infty) \big) \\ 
\big[ \image\; \Gfunc ( p-\mu-\ee < \infty) \big] &= d \Hfunc \big( [-\mu-\ee, \infty) \big).
\end{align*}
By Lemma \ref{lem:corners} along with the above substitutions, we have
$$\sum_{J \in \square_{\mu}[0,\infty)}  \tilde \Hfunc (J) = \sum_{J \in \square_{\mu} I} \tilde{\Ffunc}(J)$$
$$\sum_{J \in \square_{\mu+\ee}[0,\infty)} \tilde \Hfunc (J) = \sum_{J \in \square_{\mu+\ee} I } \tilde{G}(J).$$ 
By the inclusion $\square_{\mu} [0,\infty) \subseteq \square_{\mu+\ee} [0,\infty) $ along with 
Proposition \ref{prop:positivity}, we have 
$$\sum_{J \in \square_{\mu} [0,\infty)} \tilde \Hfunc (J) \preceq 
\sum_{J \in \square_{\mu+\ee} [0,\infty)} \tilde \Hfunc (J).$$
This proves the statement.

%
%
\end{proof}

\begin{defn}
The \define{injectivity radius}
of a finite set $S = \{ s_1 < s_2 < \cdots < s_k < \infty\}$ is
$$\rho := \min_{1 < i \leq k} \frac{s_{i}-s_{i-1}}{2}.$$
Note that if a persistence module $\Ffunc$ is $S$-constructible and $I \in \Dgm(S)$, then
	$$\tilde \Ffunc (I) = \sum_{J \in \square_\rho I} \tilde \Ffunc (J).$$
\end{defn}

\begin{lem}[Easy Bijection]
\label{lem:bijection}
Let $\Ffunc$ be an $S$-constructible persistence module and $\rho > 0$ the injectivity 
radius of $S$.
If $\Gfunc$ is a second constructible persistence module such that $d_I(\Ffunc,\Gfunc) < \rho/2$, then 
$d_B(\tilde \Ffunc, \tilde \Gfunc) \leq d_I(\Ffunc,\Gfunc)$.
\end{lem}

\begin{proof}
Let $\ee = d_I (\Ffunc, \Gfunc)$.
Choose a sufficiently small $\mu > 0$ such that $\mu + 2 \ee < \rho$.
We construct a matching $\gamma_\mu : \Dgm \times \Dgm \to \Ggroup(\Ccat)$ such that
	\begin{align}
	\tilde \Ffunc(I) & = \sum_{J \in \Dgm} \gamma_\mu(I,J) \textrm{ for all }I \in \Dgm \backslash \Delta 
	\label{eq:forward} \\
	\tilde \Gfunc(J) & = \sum_{I \in \Dgm} \gamma_\mu(I,J) \textrm{ for all } J \in \Dgm \backslash \Delta 
	\label{eq:backward} .
	\end{align}

Fix an  $I \in \Dgm(S)$. 
By Lemma \ref{lem:box}, 
$$\tilde \Ffunc(I) = \sum_{J \in \square_{\mu} I} \tilde \Ffunc (I) \leq \sum_{J \in \square_{\mu + \ee} I} \tilde \Gfunc(J) \leq \sum_{J \in \square_{\mu + 2\ee} I}\tilde \Ffunc(J) = \tilde \Ffunc(I).$$ 
Let $\gamma_\mu(I,J) := \tilde \Gfunc(J)$ for all $J \in \square_{\mu + \ee}(I)$.
Repeat for all $I \in \Dgm(S)$.
Equation \ref{eq:forward} is satisfied.
We now check that $\gamma_\mu$ satisfies Equation \ref{eq:backward}.
Fix an interval $J = [p,q)$ and suppose $\tilde \Gfunc (J) \neq 0$.
If $q-p > \ee$, then by Lemma \ref{lem:box}
	$$\tilde \Gfunc(J) = \sum_{I \in \square_{\mu} J} \tilde \Gfunc (I) \preceq \sum_{I \in \square_{\mu+\ee} J} 
	\tilde \Ffunc(J).$$
This means $\gamma_\mu(I, J) \neq 0$ for some $I \in \square_{\mu+\ee} J$.
If $q-p \leq \ee$, then we match $J$ to the diagonal. 
That is, we let $\gamma_\mu \big ( [\sfrac{q-p}{2}, \sfrac{q-p}{2} ), J \big) := \tilde \Gfunc(J)$.

By construction, $||\gamma_\mu|| \leq \mu + \ee $ for all $\mu > 0$ sufficiently small. Therefore $d_B ( \tilde \Ffunc, \tilde \Gfunc ) \leq \ee = d_I(\tilde \Ffunc, \tilde \Gfunc)$. 
\end{proof}

We are now ready to prove our main result.

\begin{thm}[Bottleneck Stability]
Let $\mathcal{C}$ be a skeletally small abelian category and $\Ffunc, \Gfunc:  \Rspace \to \mathcal{C}$ 
two constructible persistence modules.
Then $d_B \big( \tilde \Ffunc, \tilde \Gfunc \big) \leq d_I (\Ffunc, \Gfunc)$
where $\tilde \Ffunc$ and $\tilde \Gfunc$ are the persistence diagrams of $\Ffunc$ and $\Gfunc$,
respectively.
\end{thm}

\begin{proof}
Let $\ee = d_I(\Ffunc, \Gfunc)$. 
By Proposition \ref{prop:interpolation}, there is a one parameter family of constructible persistence modules 
$\{\Kfunc_t\}_{t \in [0,1]}$ such that
$d_I(\Kfunc_t, \Kfunc_s) \leq \ee |t-s|$, $\Kfunc_0 \cong \Ffunc$, and $\Kfunc_1 \cong \Gfunc$. 
Each $\Kfunc_t$ is constructible with respect to some set $S_t$, and each set $S_t$ has
an injectivity radius $\rho_t > 0$.
For each time $t \in [0,1]$, consider the open interval
$$U(t) = ( t-\sfrac{\rho_t}{4\ee},t+\sfrac{\rho_t}{4\ee} ) \cap [0,1]$$
By compactness of $[0,1]$, there is a finite set
$Q = \{ 0 = t_0 < t_1 < \cdots < t_n = 1 \}$
such that $\cup_{i=0}^n U(t_i) = [0,1]$.
We assume that $Q$ is minimal, that is, there does not exists a pair $t_i, t_j \in Q$ such that
$U(t_i) \subseteq U(t_j)$.
If this is not the case, simply throw away $U(t_i)$ and we still have a covering of $[0,1]$.
As a consequence, for any consecutive pair $t_i < t_{i+1}$, we have $U(t_i) \cap U(t_{i+1}) \neq \emptyset$.
This means
$$t_{i+1} - t_i \leq \frac{1}{4\ee} (\rho_{t_{i+1}}+\rho_{t_i}) \leq \frac{1}{2\ee}
\max\{\rho_{t_{i+1}},\rho_{t_i}\}$$
and therefore $d_I( \Kfunc_{t_{i}}, \Kfunc_{t_{i+1}}) \leq \frac{1}{2} \max \{\rho_{t_{i}},\rho_{t_{i+1}} \}$. 
By Lemma \ref{lem:bijection}, 
$$ d_B \big( \tilde \Kfunc_{t_{i}},\tilde \Kfunc_{t_{i+1}} \big) \leq 
d_I \big( \Kfunc_{t_{i}}, \Kfunc_{t_{i+1}}),$$ 
for all $1\leq i \leq n-1$. 
Therefore 
$$d_B( \tilde \Ffunc, \tilde \Gfunc \big) \leq 
\sum_{i=0}^{n-1} d_B \big( \tilde \Kfunc_{t_{i}},\tilde \Kfunc_{t_{i+1}} \big) \leq
\sum_{i=0}^{n-1} d_I \big( \Kfunc_{t_{i}}, \Kfunc_{t_{i+1}}) \leq \ee.$$  
\end{proof}

\newpage 

{\bibliography{ref.bib}}
\bibliographystyle{alpha}

\end{document}